\documentclass[12pt,leqno,twoside]{amsart}

\usepackage{amssymb,amsmath,amsthm,soul,color}

\usepackage[noadjust]{cite}

\usepackage{t1enc}

\usepackage[utf8]{inputenc}
\usepackage{indentfirst}

\usepackage{todonotes}

\usepackage[left=2.5cm, right=2.5cm, top=2cm, bottom=2cm]{geometry}

\usepackage{url}
\usepackage{soul}

\usepackage{mathrsfs}

\usepackage{enumitem,graphicx}
\usepackage[colorlinks=true,urlcolor=blue,
citecolor=red,linkcolor=blue,linktocpage,pdfpagelabels,
bookmarksnumbered,bookmarksopen]{hyperref}
\usepackage[metapost]{mfpic}
\usepackage[normalem]{ulem}

\newtheorem{Th}{Theorem}[section]
\newtheorem{Prop}[Th]{Proposition}
\newtheorem{Lem}[Th]{Lemma}

\newtheorem{Rem}[Th]{Remark}
\newtheorem{Ex}[Th]{Example}

\newcommand{\eps}{\varepsilon}

\newcommand{\R}{\mathbb{R}}

\newcommand{\Z}{\mathbb{Z}}

\newcommand{\cB}{{\mathcal B}}
\newcommand{\cC}{{\mathcal C}}
\newcommand{\cD}{{\mathcal D}}
\newcommand{\cE}{{\mathcal E}}
\newcommand{\cF}{{\mathcal F}}

\newcommand{\cH}{{\mathcal H}}
\newcommand{\cI}{{\mathcal I}}
\newcommand{\cJ}{{\mathcal J}}

\newcommand{\cL}{{\mathcal L}}
\newcommand{\cM}{{\mathcal M}}
\newcommand{\cN}{{\mathcal N}}

\newcommand{\cP}{{\mathcal P}}

\newcommand{\cS}{{\mathcal S}}

\newcommand{\cU}{{\mathcal U}}

\newcommand{\Ga}{\Gamma}

\newcommand{\weakto}{\rightharpoonup}

\numberwithin{equation}{section}

\newcommand{\curl}{\nabla \times}
\newcommand{\divv}{\mathrm{div}\,}
\renewcommand{\div}{\mathrm{div}\,}

\newcommand{\triple}[1]{{\left\vert\kern-0.25ex\left\vert\kern-0.25ex\left\vert #1 
    \right\vert\kern-0.25ex\right\vert\kern-0.25ex\right\vert}}

\begin{document}

\title{Generalized linking-type theorem with applications to strongly indefinite problems with sign-changing nonlinearities}

\author[F. Bernini]{Federico Bernini}
	\address[F. Bernini]{\newline\indent
			Dipartimento di Matematica e Applicazioni,
			\newline\indent 
			Università degli Studi di Milano-Bicocca,
			\newline\indent 
			via Roberto Cozzi 55, I-20125, Milano, Italy
	}
	\email{\href{mailto:f.bernini2@campus.unimib.it}{f.bernini2@campus.unimib.it}}

\author[B. Bieganowski]{Bartosz Bieganowski}
	\address[B. Bieganowski]{\newline\indent  	
	Institute of Mathematics, \newline\indent
	Polish Academy of Sciences, \newline\indent
	ul. \'Sniadeckich 8, 00-656 Warsaw, Poland\newline\indent
	and \newline\indent
	Faculty of Mathematics and Computer Science,		\newline\indent 	Nicolaus Copernicus University, \newline\indent ul. Chopina 12/18, 87-100 Toru\'n, Poland}	
\email{\href{mailto:bbieganowski@impan.pl}{bbieganowski@impan.pl}, \href{mailto:bartoszb@mat.umk.pl}{bartoszb@mat.umk.pl}}

%
%
%
%

\maketitle

\pagestyle{myheadings} \markboth{\underline{F. Bernini, B. Bieganowski}}{
		\underline{Generalized linking-type theorem with applications to strongly indefinite problems with sign-changing nonlinearities}}

\begin{abstract} 

We show a linking-type result which allows us to study strongly indefinite problems with sign-changing nonlinearities. We apply the abstract theory to the singular Schr\"odinger equation
$$
-\Delta u + V(x)u + \frac{a}{r^2} u = f(u) - \lambda g(u), \quad x = (y,z) \in \R^K \times \R^{N-K}, \ r = |y|,
$$
where
$$
0 \not\in \sigma \left( -\Delta + \frac{a}{r^2} + V(x) \right).
$$
As a consequence we obtain also the existence of solutions to the nonlinear curl-curl problem. \\

\end{abstract}

\section{Introduction}

In this paper we are interested in an abstract setting which allows us to study strongly indefinite problems with sign-changing nonlinearities. Consider a general, real Hilbert space $(X, \| \cdot \|)$ and a nonlinear functional $\cJ : X \rightarrow \R$ of $\cC^1$-class. We are looking for nontrivial critical points of $\cJ$, i.e. points $u \in X \setminus \{0\}$ with $\cJ'(u) = 0$. If $\cJ$ is sequentially weak-to-weak* continuous, then the problem usually reduces to finding a \textit{Palais-Smale sequence} $(u_n) \subset X$:
$$
(\cJ(u_n)) \subset \R \mbox{ is bounded} \quad \mbox{and} \quad \cJ'(u_n) \to 0 \mbox{ in } X^*
$$
or a \textit{Cerami sequence} $(u_n) \subset X$:
$$
(\cJ(u_n)) \subset \R \mbox{ is bounded} \quad \mbox{and} \quad (1+\|u_n\|)\cJ'(u_n) \to 0 \mbox{ in } X^*.
$$
Having such a sequence, from the sequential weak-to-weak* continuity of $\cJ'$, one can immediately see that any weak limit point of $(u_n)$ is a critical point of $\cJ$. Hence the emphasis is on finding such a sequence. Suppose, for simplicity, that $\cJ$ is sufficiently regular and that $X = X^+ \oplus X^-$ has an orthogonal splitting such that the second variation $\cJ''(0)[u][u]$ is positive definite on $X^+$ and negative definite on $X^-$. If $X^- = \{0\}$ we say that the problem is positive definite, otherwise we say that the problem is strongly indefinite. Suppose that $\cJ$ is of the form
$$
\cJ(u) = \frac12 \|u^+\|^2 - \frac12 \|u^-\|^2 - \cI(u), \quad u = u^+ + u^- \in X^+ \oplus X^-,
$$
where $\cI$ is the nonlinear (usually super-quadratic at infinity) part of $\cJ$.

In the positive definite case one can use the mountain pass geometry (introduced by Ambrosetti and Rabinowitz, see \cite{AR}) or the Nehari manifold method (proposed in \cite{Nehari}) to find a Palais-Smale sequence (\cite{BiegMed}) or a Cerami sequence (\cite{Bi}), even for functionals with sign-changing nonlinear part $\cI$. A first important result that allows a variational treatment of strongly indefinite problems is the linking theorem proved by Rabinowitz in 1978 (see \cite{R}) in the case when one of the spaces $X^+$, $X^-$ have a finite dimension. The result has been later generalized by Kryszewski and Szulkin in 1997 (see \cite{KS}) allowing to manage the infinite dimension of both spaces $X^+$ and $X^-$. A second possible approach is to use the Nehari-Pankov manifold (provided by Pankov, see \cite{P}), which has been successfully applied by Szulkin and Weth (see \cite{SzW}) in the case where $X^+$, $X^-$ have infinite dimension. Both of this approaches require that $\cI(u) \geq 0$, cf. \cite{BM, KS, M2, SzW}.

For this reason, we are interested in providing a result in which this last request can be drop off.
Our result is a linking-type approach which may be viewed as a modification of abstract results from \cite{KS, M2} and a generalization of \cite{CW}. Moreover, we also show that, under certain condition (e.g. $\cI(u) \geq 0$), $\cJ(u_n)$ can be bounded by the infimum of $\cJ$ on the Nehari-Pankov manifold and therefore we can recover the existence of solutions to a certain class of equations considered in \cite{KS, M2}. 

We will show an application of the abstract theorem to the Schr\"odinger equation
$$
-\Delta u + V(x)u + \frac{a}{r^2} u = f(u) - \lambda g(u), \quad x = (y,z) \in \R^K \times \R^{N-K}, \ r = |y|,
$$
which arises from the mathematical physics. In particular, one can obtain it when looking for time-harmonic electric fields being solutions of a particular nonlinear Maxwell equation or when looking for standing waves for the time-dependent Schr\"odinger equation.

The system of Maxwell equations is of the form
$$
\left\{ \begin{array}{l}
\curl \cH = \cJ + \frac{\partial \cD}{\partial t} \\
\div (\cD) = \rho \\
\frac{\partial \cB}{\partial t} + \curl \cE = 0 \\
\div (\cB) = 0,
\end{array} \right.
$$
where $\cE$ is the electric field, $\cB$ is the magnetic field, $\cD$ is the electric displacement field and $\cH$ denotes the magnetic induction. Moreover $\cJ$ denotes the electric current intensity and $\rho$ the electric charge density. We consider also the following constitutive relations
$$
\begin{cases}
\cD=\eps \cE + \cP\\
\cH=\frac{1}{\mu}\cB - \cM,
\end{cases}
$$
where $\cP$ is the polarization and $\cM$ is the magnetization. In the absence of charges, currents and magnetization, and assuming that $\mu \equiv 1$, where $\mu$ is the permeability of the medium, we obtain the time-dependent equation (see e.g. \cite{BaMe1})
$$
\curl \left( \curl \cE \right) + \varepsilon \frac{\partial^2 \cE}{\partial t^2} = - \frac{\partial^2 \cP}{\partial t^2},
$$
where $\varepsilon$ is the permittivity of the medium.We look for a time-harmonic field $\cE = \mathbf{E}(x) \cos(\omega t)$. Moreover, we suppose that the nonlinear polarization $\cP$ is of the form 
$$
\cP = \chi\left( \langle |\cE|^2  \rangle \right) \cE, 
$$
i.e. the scalar dielectric susceptibility $\chi$ depends only on the time average 
$$
\langle |\cE|^2 \rangle = \frac{1}{T} \int_0^T | \cE(x,t)|^2 \, dt = \frac12 |\mathbf{E}|^2
$$ 
of the intensity of $\cE$, where $T = \frac{2\pi}{\omega}$. Hence, $\cP = \mathbf{P}(\mathbf{E}(x)) \cos(\omega t)$, where $\mathbf{P}(\mathbf{E}) = \chi\left( \frac12 |\mathbf{E}|^2 \right) \mathbf{E}$. This ansatz lead to 
\begin{equation}\label{eq:maxwell}
\curl (\curl \mathbf{E}) + V(x)\mathbf{E} = h(\mathbf{E}), \quad x \in \R^3
\end{equation}
with $V(x) = - \omega^2 \varepsilon(x)$ and $h(\mathbf{E}(x)) = \mathbf{P}(\mathbf{E}(x)) \omega^2$. For media with Kerr effect, strong electric fields $\cE$ of high intensity cause the refractive index to vary quadratically and then $\cP$ has the form
$$
\cP( t, x) = \alpha(x) \langle |\cE|^2 \rangle \cE,
$$
see \cite{Nie, S1}. Assuming that $\alpha(x) \equiv \alpha$ is a constant, we get $\mathbf{P}(\mathbf{E}(x)) = \frac{\alpha}{2} | \mathbf{E}(x)|^2 \mathbf{E}(x)$. In the paper we are interested in the more general case, where the polarization may consists of two competing terms, e.g. $\mathbf{P}(\mathbf{E}) = |\mathbf{E}|^{p-2} \mathbf{E} - |\mathbf{E}|^{q-2} \mathbf{E}$.

Moreover, we study the total electromagnetic energy given by
\begin{equation}
\label{eq:electroen}
	\cL(t):= \frac12 \int_{\R^3} \cE\cD + \cB\cH \, dx.
\end{equation}
We show that $\cL(t)$ is finite and constant (does not depend on $t$) for the solution we find in Theorem \ref{th:main2}.
For more detailed physical background see e.g. \cite{A, DLPSW, M1, S1, S2}. 

Note that the kernel of $\curl \curl$ has an infinite dimension, because $\curl (\nabla \varphi) = 0$ for any $\varphi \in \cC_0^\infty (\R^3)$. Hence, the energy functional associated with \eqref{eq:maxwell} is strongly indefinite. Moreover its derivative is not weak-to-weak* continuous and every nontrivial critical point has infinite Morse index. Hence, we will consider the cylindrically symmetric setting and reduce the curl-curl problem to the Schr\"odinger equation. 

The problem \eqref{eq:maxwell} in a bounded domain $\Omega \subset \R^3$ in a case when the domain is surrounded by a perfect conductor, i.e.
$$
\nu \times \mathbf{E} = 0 \quad \mbox{on } \partial \Omega
$$
was studied in a series of papers \cite{BaMe1, BaMe2, BaMe3}. Under the same boundary condition an eigenvalue problem was studied in \cite{Y}. See also \cite{HL}.

Looking for classical solutions of the form (see e.g. \cite{BDPR, Z})
\begin{equation}\label{eq:form}
\mathbf{E}(x) = \frac{u(r, x_3)}{r} \left( \begin{array}{c} -x_2 \\ x_1 \\ 0 \end{array} \right), \quad r = \sqrt{x_1^2+x_2^2}
\end{equation}
to \eqref{eq:maxwell} leads to 
\begin{equation}\label{eq:schroedinger}
-\Delta u + V(x)u + \frac{a}{r^2} u = f(u) - \lambda g(u), \quad x = (y,z) \in \R^K \times \R^{N-K}, \ r = |y|,
\end{equation}
with $N=3$, $K=2$, $a=1$, where $\Delta = \frac{\partial ^2}{\partial r^2} + \frac{1}{r} \frac{\partial}{\partial r} + \frac{\partial^2}{\partial x_3^2}$ is the 3-dimensional Laplacian operator in cylindrically symmetric coordinates $(r, x_3)$, and nonlinear terms are described by the following relation
$$
h(\mathbf{E}) = f(\alpha)w - \lambda g(\alpha)w,
$$
where $\mathbf{E} = \alpha w$ for some $w \in \R^3$, $|w|=1$, $\alpha \in \R$ and $h$ is the nonlinear term in \eqref{eq:maxwell}. This equivalence also holds for weak solutions (see \cite{Bi, GMS}).

We would like to point out that \eqref{eq:schroedinger}, where $a > -\frac{(K-2)^2}{4}$ and $N > K \geq 2$, is also of particular interest on its own.  The problem naturally appears when looking for standing waves $\Psi(x, t) = u(x)e^{-\mathbf{i}\omega t}$ for the time-dependent nonlinear Schr\"odinger equation (see e.g. \cite{BBR}) of the form
$$
\mathbf{i} \frac{\partial \Psi}{\partial t} = -\Delta \Psi + \left( V(x) + \frac{a}{r^2} u + \omega \right) \Psi - f(|\Psi|) + \lambda g(|\Psi|).
$$ 
In \cite{BBR} the authors found a nontrivial and nonnegative solution to \eqref{eq:schroedinger} with $a=1$, $g \equiv 0$ and $V \equiv 0$. This problem was also studied in \cite{GMS} with $V \equiv 0$ and $g \equiv 0$, and the authors investigate the existence and multiplicity of solutions. 

The Schr\"odinger equation \eqref{eq:schroedinger} appears in nonlinear optics, where photonic crystals admitting nonlinear effects are studied (\cite{Kuch}). Then the nonlinearity is responsible for the polarization in a photonic crystal and e.g. in a self-focusing Kerr-like media one has $f(u) = |u|^{2}u$ and $g\equiv 0$ (\cite{Bur, Slu}). In the case $f(u) = |u|^{p-2}u$ and $g(u) = |u|^{q-2}u$, $\lambda > 0$ and $p > q$, we deal with a mixture of self-focusing and defocusing materials. Such nonlinearities in Schr\"odinger equations were studied in the positive definite case in \cite{BiegMed}.

Strongly indefinite Schr\"odinger-type equations of the form
\begin{equation*}
	-\Delta u + V(x)u = f(x,u)
\end{equation*}
with the associated energy functional $\cJ:X \to \R$ have been of great interest in recent years, starting with the works of \cite{AlLi}, \cite{Jeanjean}. Some years later, Troestler and Willem (see \cite{TrWi} proved the existence of a non-trivial solution through a linking theorem requiring that the nonlinearity satisfy the so-called \textit{Ambrosetti-Rabinowitz condition}
$$
0 < \mu F(x,u) \leq f(x,u)u, \quad \mbox{where } \mu > 2,
$$ 
and that the associated functional $\cJ$ is of $\cC^2$-class. The linking geometry was demonstrated by exploiting a generalization of the Smale's degree for Fredholm maps. Then, Kryszewski and Szulkin in \cite{KS}, through the abstract result already cited above, proved the existence of (at least) a nontrivial solution for the functional of $\cC^1$-class and nonlinearities satisfying the Ambrosetti-Rabinowitz condition.

In 2005, Pankov (see \cite{P}) proved that the equation admits a nontrivial weak solution that is continuous and that exponentially decays to infinity. For this purpose, he developed a method consisting on the reduction of the problem to a problem restricted to the following set
\[
	\cN=\left\{u \in X \setminus X^- : \cJ'(u)(u) = 0 \text{ and } \cJ'(u)(v) = 0 \text{ for every } v \in X^- \right\},
\] 
which contains all the critical points of $\cJ$. Moreover $\cN$ is of $\cC^1$-class and is a natural constraint for the energy functional $\cJ$.

This technique was then developed by Szulkin and Weth in \cite{SzW} by removing an hypothesis on the growth of the first derivative (which is stronger than the Ambrosetti-Rabinowitz condition and the monotonicity assumption from \cite{SzW}). Without sufficient regularity of the nonlinearity, the manifold $\cM$ needs not be anymore of $\cC^1$-class. Therefore Pankov's method does not work. To overcome this problem, Szulkin and Weth construct a homemorphism of $\cM$ and the unit sphere in $X^+$. Moreover, this homeomorphism preserves the $\cC^1$-class of the functional and the problem can be reduced to the sphere being a $\cC^1$-manifold.

However, in all the works cited so far it is required that $\cI\geq 0$, which is not our case. The sign-changing nonlinearity does not satisfy the Ambrosetti-Rabinowitz condition. Furthermore, the double indefinite nature of the problem, given by the operator and the sign-changing nonlinearity, do not even allow us to use the Pankov or Szulkin-Weth strategy, since it is not clear how to construct a homeomorphism between the set $\cN$ and the unit sphere in $X^+$.

Consequently, we cannot even use the approach provided by Mederski and the second author in \cite{BiegMed}, where the Szulkin-Weth approach is used to deal with sign-changing nonlinearities in the positive definite case. As an application, they find a ground state solution to the problem
\[
	-\Delta u + V(x)u = f(x,u) - \Gamma (x)|u|^{q-2}u
\]
where the spectrum of the Schr\"odinger operator is positive. Furthermore, in the spirit of Pankov, they proved that this solution is continuous and vanishes exponentially at infinity. Regarding the existence of solutions, our result can be also applied in this case with $\Gamma (x)\equiv \lambda$, $f(x,u)=f(u)$ and $g(u) = |u|^{q-2}u$. 

We end this Section by listing the assumptions we made on the equation \eqref{eq:schroedinger} and providing some examples of nonlinearities $f$ and $g$. Then we state our results on the existence of solutions to \eqref{eq:schroedinger} and \eqref{eq:maxwell}. Hereafter $O(K)$ denotes the group of real and orthogonal $K \times K$-matrices and its action on the euclidean space is given by the multiplication. 

We assume that the potential $V$ satisfies
\begin{itemize}
\item[(V)] $V \in L^\infty (\R^N)$ is $O(K) \times \{I_{N-K}\}$-invariant, $\Z^{N-K}$-periodic in $z$ and
\begin{equation}\label{eq:zero-in-gap}
0 \not\in \sigma \left( -\Delta + \frac{a}{r^2} + V(x) \right) \quad \mbox{and} \quad \sigma \left( -\Delta + \frac{a}{r^2} + V(x) \right) \cap (-\infty,0) \neq \emptyset.
\end{equation}
\end{itemize}
For examples of sign-changing potentials satisfying (V) see e.g. \cite{BDPR, Bi}. The assumption (V) implies that the energy functional associated with \eqref{eq:schroedinger} is strongly indefinite. Hence, in the paper we are interested in the general, variational setting which allows us to study strongly indefinite problems with sign-changing nonlinearities, including \eqref{eq:schroedinger}. 

To show the existence of a nontrivial solution we impose the following assumptions. In what follows we use $\lesssim$ to denote the inequality up to a multiplicative constant.

\begin{enumerate}
\item[(F1)] $f: \R \to \R$ is odd, continuous and there is $2<p<2^* := \frac{2N}{N-2}$ such that
\[
	|f(u)| \lesssim 1+|u|^{p-1} \text{ for all } u \in \R.
\]
\item[(F2)] $f(u)=o(|u|)$ as $u \to 0$.
\item[(F3)] There is $2 < q < p$ such that $F(u)/|u|^q \to +\infty$ as $|u| \to +\infty$, where $F(u)=\int_0^u f(s) \, dx$ and $F(u) \geq 0$ for all $u \in \R$.
\item[(F4)] $u \mapsto f(u)/|u|^{q-1}$ is nondecreasing in $(-\infty,0)$ and on $(0,+\infty)$.
\item[(F5)] There is $\rho > 0$ such that $|u|^{p-1} \lesssim |f(u)| \lesssim |u|^{p-1}$ for $|u| \geq \rho$.
\item[(G1)] $g: \R \to \R$ is odd, continuous such that
\[
	|g(u)| \lesssim 1+|u|^{q-1} \text{ for all } u \in \R.
\]
\item[(G2)] $g(u)=o(|u|)$ as $u \to 0$.
\item[(G3)] $u \mapsto g(u)/|u|^{q-1}$ is nonincreasing in $(-\infty,0)$ and on $(0,+\infty)$ and there holds
\[
	g(u)u \geq 0 \text{ for all } u \in \R.
\]
\end{enumerate}

We provide several examples of nonlinearities satisfying the foregoing assumptions.

\begin{Ex}
In what follows, $2 < q < p < 2^*$.
\begin{enumerate}
\item[(i)] It is clear that $f(u) = |u|^{p-2}u$ and $g(u) = |u|^{q-2}u$ satisfy (F1)--(F5), (G1)--(G3).
\item[(ii)] Let $\rho > 0$. We will check that 
$$
f(u) = \left\{ \begin{array}{ll}
|u|^{q-2} u \log \left(1+|u|^{p-q}\right) & \quad \mbox{for } |u| < \rho \\
C \left(1 + \arctan |u| \right) |u|^{p-2} u & \quad \mbox{for } |u| \geq \rho
\end{array} \right.
$$
satisfies (F1)--(F5), where $C > 0$ is chosen so that $f$ is continuous on $\R$. Indeed, (F1) is clear. To get (F2) we see that
$$
\lim_{u\to 0} \frac{f(u)}{u} = \lim_{u\to 0} |u|^{q-2} \log \left(1+|u|^{p-q}\right) = 0.
$$
Using the L'H\^{o}pital's rule we get
$$
\lim_{|u|\to+\infty} \frac{F(u)}{|u|^q} = \lim_{|u|\to+\infty} \frac{f(u)}{q|u|^{q-2}u} = \frac{C}{q}  \lim_{|u|\to+\infty}\left(1 +  \arctan |u| \right) |u|^{p-q} = +\infty.
$$
Moreover, for $u \geq 0$, $F(u) = \int_0^u f(s) \, ds \geq 0$ and, since $f$ is odd, $F(-u)=F(u) \geq 0$. Thus (F3) holds. To show (F4) we note that
$$
\frac{f(u)}{|u|^{q-1}} = \left\{ \begin{array}{ll}
u \log \left(1+|u|^{p-q}\right) & \quad \mbox{for } |u| < \rho \\
C \left(1 + \arctan |u| \right) |u|^{p-q} u & \quad \mbox{for } |u| \geq \rho
\end{array} \right.
$$
is nondecreasing on $(-\infty,0)$ and on $(0,+\infty)$. To get (F5) we note that for $|u| \geq \rho$ the following estimates hold true
$$
C |u|^{p-2}u \leq C(1+\arctan |u|) |u|^{p-2}u \leq C \left(1 + \frac{\pi}{2} \right) |u|^{p-2}u.
$$
\item[(iii)] We will show that $g(u) = \frac{|u|^{q-2}u}{1+e^{|u|}}$ satisifes (G1)--(G3). (G1) is clear, to get (G2) we compute
$$
\lim_{u\to 0} \frac{g(u)}{u} = \lim_{u\to 0} \frac{|u|^{q-2}}{1+e^{|u|}} = 0.
$$
To see (G3) we observe that $\frac{g(u)}{|u|^{q-1}} = \frac{\mathrm{sgn}\, (u)}{1+e^{|u|}}$ is nonincreasing on $(-\infty,0)$ and on $(0,+\infty)$. Moreover $g(u)u = \frac{|u|^{q}}{1+e^{|u|}} \geq 0$ for $u \in \R$.
\item[(iv)] Similarly one can compute that $g(u) = \frac{|u|^{q-2}u}{1+\arctan |u|}$ satisfies (G1)--(G3).
\end{enumerate}
\end{Ex}


\begin{Th}\label{th:main1}
Suppose that (V), (F1)--(F5), (G1)--(G3) hold. If $\lambda > 0$ and $\rho > 0$ in (F5) are sufficiently small, there is a nontrivial soluton to \eqref{eq:schroedinger}.
\end{Th}

As a consequence of Theorem \ref{th:main1} and the equivalence of solutions (see Theorem \ref{th:equiv}) we obtain also the existence result for the Maxwell problem \eqref{eq:maxwell}.

\begin{Th}\label{th:main2}
Under assumptions of Theorem \ref{th:main1} there is a nontrivial solution to \eqref{eq:maxwell} of the form \eqref{eq:form}.
\end{Th}

The structure of the paper is the following. In Section \ref{sect:abstract} we introduce an abstract setting based on a linking-type approach, which allows us to find a Cerami sequence which is bounded away from $0$ (see also Section \ref{sect:3}). The Section \ref{sect:functional} is devoted to the functional setting for equations \eqref{eq:maxwell} and \eqref{eq:schroedinger}. In Section \ref{sect:5} we study the boundedness of a Cerami sequence. The last, sixth Section contains proofs of Theorems \ref{th:main1} and \ref{th:main2}, and the property of the total electromagnetic energy (Proposition \ref{prop:energy}).

In what follows $C$ denotes a generic, positive constant which may vary from one line to another. Moreover $| \cdot |_k$ denotes the usual $L^k$-norm.

\section{\texorpdfstring{$\tau$}{tau}-topology and critical point theory}\label{sect:abstract}

Our approach is mainly based on \cite{CW, KS, M2}. Let $(X, \| \cdot \|)$ be a Hilbert space. Assume that there is an orthogonal splitting $X = X^+ \oplus X^-$. It is clear that every $u \in X$ has a unique decomposition $u = u^+ + u^-$, where summands satisfy $u^\pm \in X^\pm$. In order to state the critical point theory we introduce a new topology $\tau$ in the space $X$, see \cite{KS}. Let $(e_k)_{k=1}^\infty \subset X^-$ be a complete orthonormal sequence in the space $X^-$. Then we define a norm $\triple{\cdot}$ in $X$ by
$$
\triple{u} := \max \left\{ \| u^+ \|, \sum_{k=1}^\infty \frac{1}{2^{k+1}} \left| \langle u^-, e_k \rangle \right| \right\}.
$$
Let $\tau$ denote the topology on $X$ generated by $\triple{\cdot}$. We note that $\tau$ is weaker than the topology generated by the norm $\| \cdot \|$ and that the following inequalities hold
$$
\|u^+\| \leq \triple{u} \leq \|u\|.
$$
We also recall that for bounded sequences $(u_n) \subset X$ the following equivalence holds true (see e.g. \cite[Remark 2.1(iii)]{KS}, \cite[Chapter 6]{W})
$$
u_n \stackrel{\tau}{\to} u \quad \Longleftrightarrow \quad u_n^+ \to u^+ \mbox{ and } u_n^- \weakto u^-. 
$$

Let $\cJ : X \rightarrow \R$ be a nonlinear functional. For $u \in X \setminus X^-$ and $R > r > 0$ we introduce the following sets:
\begin{align*}
\cS^+_r &:= \{ u^+ \in X^+ \ : \ \|u^+\| = r \} \\
M(u) &:= \{ tu + v^- \ : \ v^- \in X^-, \ t \geq 0, \, \|tu + v^- \| \leq R \}.
\end{align*}
It is clear that $M(u) \subset \R_+ u^+ \oplus X^-$, where $\R_+ := [0,\infty)$, is a submanifold with the boundary
$$
\partial M(u) = \{ v^- \in X^- \ : \ \| v^- \| \leq R \} \cup \{tu + v^- \ : \ v^- \in X^-, \ t > 0, \ \|tu + v^- \|=R \}.
$$
Moreover, for $\alpha \leq \beta$ we introduce the following sets
$$
\cJ^\beta := \{ u \in X \ : \ \cJ(u) \leq \beta \}, \ \cJ_\alpha := \{ u \in X \ : \ \alpha < \cJ(u) \}, \ \cJ_\alpha^\beta := \cJ_\alpha \cap \cJ^\beta.
$$

We are working under the following assumptions.

\begin{itemize}
\item[(A1)] $\cJ$ is of $\cC^1$-class and $\cJ(0) = 0$.
\item[(A2)] $\cJ'$ is sequentially weak-to-weak* continuous.
\end{itemize}

Let $\cP \subset X \setminus X^-$ be a nonempty set. We assume that
\begin{itemize}
\item[(A3)] there are $\delta > 0$ and $r > 0$ such that for every $u \in \cP$ there is radius $R=R(u) > r$ with
$$
\inf_{\cS_r^+} \cJ > \max \left\{ \sup_{\partial M(u)} \cJ, \sup_{\triple{v} \leq \delta} \cJ(v) \right\}.
$$
\end{itemize}

Let
$$
\cN := \left\{ u \in X \setminus X^- \ : \ \cJ'(u)(u) = 0, \ \cJ'(u)(v) = 0 \mbox{ for all } v \in X^- \right\}
$$
denote the Nehari-Pankov manifold and we introduce its subset
$$
\cN_\cP := \cN \cap \cP.
$$

We also consider the following assumption, which allows us to compare the energy level in Theorem \ref{th:abstrExistence} with the infimum on $\cN_\cP$, if satisfied.
\begin{itemize}
\item[(A4)] For each $u \in \cN_\cP$, $v \in X^-$ and $t \geq 0$ there holds $\cJ(u) \geq \cJ(tu+v)$.
\end{itemize}

Let $A \subset X$ and $I :=[0,1]$. Let $h : A \times I \rightarrow X$. We consider the following conditions. 
\begin{itemize}
\item[(h1)] $h$ is $\tau$-continuous, i.e. $h(v_n, t_n) \stackrel{\tau}{\to} h(v, t)$ for $v_n \stackrel{\tau}{\to} v$ and $t_n \to t$;
\item[(h2)] $h(u,0)=u$ for $u \in A$;
\item[(h3)] $\cJ(u)\geq \cJ(h(u,t))$ for $(u,t) \in A \times I$;
\item[(h4)] for every $(u,t) \in A \times I$ there is an open - in the product topology of $(X, \tau)$ and $(I, |\cdot |)$ - neighborhood $W \subset X \times I$ of $(u,t)$ such that $\{v - h(v,s) \ : \ (v,s) \in W \cap (A \times I)\}$ is contained in a finite-dimensional subspace of $X$.
\end{itemize}
We additionally say that $h$ is \textit{admissible}, if it satisfy (h1) and (h4).

\begin{Th}\label{th:abstrExistence}
Suppose that $\cJ$ satisfy (A1)--(A3). Then there is a  Cerami sequence $(u_n) \subset X$ bounded away from zero, i.e. a sequence such that
\begin{equation}\label{Cer}
\sup_n \cJ(u_n) \leq c, \quad (1+\|u_n\|) \cJ'(u_n) \to 0 \mbox{ in } X^*, \quad \inf_n \triple{u_n} \geq \frac{\delta}{2},
\end{equation}
where
$$
c := \inf_{u \in \cP} \inf_{h \in \Gamma(u)} \sup_{u' \in M(u)} \cJ(h(u',1)) \geq \inf_{\cS_r^+} \cJ > 0
$$
and
$$
\Gamma(u) := \left\{ h \in \cC(M(u) \times [0,1]) \ : \ h \mbox{ satisfy (h1)--(h4)} \right\} \neq \emptyset.
$$
If additionally (A4) holds, $c \leq \inf_{\cN_\cP} \cJ$, where we define $\inf_\emptyset \cJ := +\infty$.
\end{Th}

In comparison to \cite{KS, M2} in our approach we don't require that $\cJ$ is $\tau$-upper semicontinuous, usually is checked by means of Fatou's lemma, which requires the nonnegativity of the nonlinear part of the energy functional $\cJ$. In place of this assumption we require a stronger inequality in (A3) than in \cite{KS, M2}. Our result can be treated also as a generalization of \cite{CW} - we find a Cerami sequence below the level $c > 0$, which can be compared with the infimum of $\cJ$ on the Nehari-Pankov manifold. In particular, using our approach, one can find a Cerami sequence on the least energy solution level and then find a ground state solution, provided that the nonlinear part of $\cJ$ is nonnegative. Therefore our approach applies to problems considered in \cite{KS, M2}, where we can also find ground state solutions, as well as to problems in \cite{CW}, where we can find a nontrivial solution.

\begin{proof}
\textbf{Step 1.} \textit{The family $\Gamma(u)$ is nonempty for $u \in \cP$.} \\
Fix $u \in \cP$ and consider $h : M(u) \times [0,1] \rightarrow X$ given by $h(v,t) = v$. It is clear that (h1), (h2) and (h3) are satisfied. Note that $v - h(v,s) = v-v = 0$, so that it is sufficient to take $W = X \times I$. Hence, $h \in \Gamma(u)$ and $\Gamma(u) \neq \emptyset$. \\
\textbf{Step 2.} $c \geq \inf_{\cS_r^+} \cJ$. \\
Fix $u \in \cP$ and $h \in \Gamma(u)$. Define the map
$$
H : M(u) \times [0,1] \rightarrow \R u^+ \oplus X^- \subset X
$$
by the formula
$$
H(v, t) := \left( \| h(v,t)^+ \| - r \right) \frac{u^+}{\|u^+\|} + h(v,t)^-.
$$
We will show that $H$ is admissible. (h1) is clear. To get (h4) we fix  the point $(v',t)$ and take the neighborhood $W$ for $h$ and this point. Then it is clear that
\begin{align*}
v-H(v,s) &= v - \left( \| h(v,s)^+ \| - r \right) \frac{u^+}{\|u^+\|} - h(v,s)^- \\
&= v^+ - \left( \| h(v,s)^+ \| - r \right) \frac{u^+}{\|u^+\|} + (v-h(v,s))^-
\end{align*}
and $\{ v - H(v,s) \ : \ (v,s) \in W \cap (A \times I) \}$ is contained in a finite-dimensional subspace. Observe that $H(v,t) = 0$ if and only if $h(v,t)^- = 0$ and $\| h(v,t)^+\| = r$, i.e. $h(v,t) \in \cS_r^+$. Suppose that for some $(v,t) \in \partial M (u) \times [0,1]$ we have $H(v,t) = 0$ or, equivalently, $h(v,t) \in \cS_r^+$. Then (h3) implies that
$$
\sup_{\partial M(u)} \cJ \geq \cJ(v) \geq \cJ(h(v,t)) \geq \inf_{\cS_r^+} \cJ,
$$ 
which contradicts (A3). Hence, $0 \not\in H(\partial M (u) \times [0,1])$. It is also easy to see that
$$
H(v,0) = v - r \frac{u^+}{\|u^+\|}.
$$
By the homotopy invariance and the existence property (\cite[Theorem 2.4(i)--(iii)]{KS}) of the degree we obtain
$$
\deg (H(\cdot, 1), M(u), 0) = \deg(H(\cdot, 0), M(u), 0) = \deg \left( I - r \frac{u^+}{\|u^+\|}, M(u), 0 \right) = 1,
$$
where $\deg$ denotes the topological degree defined in \cite{KS}. Hence, $\deg (H(\cdot, 1), M(u), 0) \neq 0$ and there is $v \in M(u)$ with $H(v,1) = 0$. It means that $h(v,1) \in \cS_r^+$ and therefore
$$
\sup_{u' \in M(u)} \cJ(h(u',1)) \geq \cJ(h(v,1)) \geq \inf_{\cS_r^+} \cJ > 0,
$$
and the proof of Step 2 is completed. \\
From this point, we will show the existence of a Cerami sequence, satisfying \eqref{Cer}, by a contradiction. Hence, suppose that for some $\varepsilon > 0$ there holds
$$
(1 + \|u\|) \| \cJ'(u)\| \geq \varepsilon
$$
for $u \in \cJ^{c+\varepsilon} \cap \left\{ u \in X \ : \ \triple{u} \geq \frac{\delta}{2} \right\}.$ Without loss of generality we assume that $\varepsilon < \inf_{\cS_r^+} \cJ.$
 \\
\textbf{Step 3.} \textit{The existence of a vector field in a neighborhood of $\cJ^{c+\varepsilon} \cap \left\{ u \in X \ : \ \triple{u} \geq \frac{\delta}{2} \right\}$ and the construction of the flow $\eta$.} \\
Set $Y := \cJ^{c+\varepsilon} \cap \left\{ u \in X \ : \ \triple{u} \geq \frac{\delta}{2} \right\}$ for the simplicity of notation. Let $u \in Y \cap B_\rho$, where $\rho > 0$ and $B_\rho := \{ u \in X \ : \ \|u\| \leq \rho \}$, and define
$$
w(u) := \frac{2 \nabla \cJ(u)}{\| \cJ'(u) \|^2}.
$$
Then
$$
\langle \nabla \cJ(u), w(u) \rangle = 2
$$
and
$$
\| w(u) \| = \frac{2}{\| \cJ'(u) \|} \leq \frac{2}{\varepsilon} (1+\|u\|).
$$
Then there is a $\tau$-open neighborhood $U_u$ of $u$ with
$$
\langle \nabla \cJ(v), w(u) \rangle > 1, \ \|w(u)\| \leq \frac{4}{\varepsilon} (1+\|v\|) \mbox{ for } v \in U_u,
$$
cf. \cite[Proposition 3.2, Remark 2.1(iii)]{KS}. Indeed, if $(u_n) \subset Y \cap B_\rho$ is a sequence with $u_n \stackrel{\tau}{\to} u$, then $u_n \weakto u$ in $X$. Hence, by (A2), $\cJ'(u_n)(\varphi) \to \cJ'(u)(\varphi)$ for $\varphi \in X$. Thus $\cJ'$ is sequentially $\tau$-continuous in $Y \cap B_\rho$. The obtained $\tau$-continuity of $\cJ'$ and the weak lower semi-continuity of the norm in $X$ imply the existence of a neighborhood $U_u$ as above.

Observe that the closed ball $B_\rho$ is bounded and convex, and therefore it is also $\tau$-closed. Thus $U_0 := X \setminus B_\rho$ is a $\tau$-open set and therefore the family $\cF := \{ U_u \}_{u \in Y \cap B_\rho} \cup \{ U_0 \}$ is a $\tau$-open covering of $Y$. Set
$$
\cU := \bigcup \cF.
$$
It is clear that $\cF$ is a $\tau$-open covering of the metric space $\cU$, which is paracompact. We may find a $\tau$-locally finite $\tau$-open refinement $\{\tilde{N}_j\}_{j \in J}$ of the covering $\cF$ of $\cU$. It is clear that 
$$
Y \subset \cU \subset \tilde{N} := \bigcup_{j \in J} \tilde{N}_j
$$
and $\tilde{N}$ is a $\tau$-open set. Let $\{ \lambda_j\}_{j \in J}$ denotes the $\tau$-Lipschitzian partition of unity subordinated to $\{ \tilde{N}_j \}_{j \in J}$. If $\tilde{N}_j \subset U_{u_j}$ for some $u_j$, we put $w_j := w(u_j)$. Otherwise, if $\tilde{N}_j \subset U_0$, we set $w_j = 0$. We put
$$
\tilde{V}(u) := \sum_{j\in J} \lambda_j(u) w_j, \quad u \in \tilde{N}.
$$
For every $u \in \tilde{N}$ the above sum is finite and therefore there is a $\tau$-open neighborhood $U_u \subset \tilde{N}$ of $u$ and $L_u$ such that $\tilde{V}(U_u)$ is contained in a finite-dimensional subspace of $X$ and
$$
\| \tilde{V}(v) - \tilde{V}(w) \| \leq L_u \triple{v-w}
$$
for $v,w \in U_u$. It is clear that
$$
\langle \nabla \cJ(u), \tilde{V}(u) \rangle \geq 0, \quad u \in \tilde{N}
$$
and
$$
\langle \nabla \cJ(u), \tilde{V}(u) \rangle > 1, \quad u \in  Y \cap B_\rho.
$$
Moreover, we compute that
$$
\left\| \tilde{V}(u) \right\| \leq \frac{4}{\varepsilon} (1 + \|u\|), \quad u \in \tilde{N}.
$$
Choose a smooth function $\chi : \R \rightarrow [0,1]$ such that $0 \leq \chi(t) \leq 1$ for $t \in \R$, $\chi(t) = 0$ for $t \leq \frac{2\delta}{3}$ and $\chi(t)=1$ for $t \geq \delta$. Then we set
$$
V(u) := \left\{ \begin{array}{ll}
\chi(\triple{u}) \tilde{V}(u), & \quad u \in \tilde{N}, \\
0, & \quad \triple{u} \leq \frac{2\delta}{3}.
\end{array} \right.
$$
Put $N := \cU \cup \{u \in X \ : \ \triple{u} < \delta \}$. Then $N$ is a $\tau$-neighborhood of $\cJ^{c+\eps} \cup (X \setminus B_\rho)$. It is clear that $V$ is locally Lipschitz and $\tau$-locally Lipschitz continuous; moreover
$$
\|V(u)\| \leq \frac{4}{\eps} (1+\|u\|), \quad \langle \nabla \cJ(u), V(u) \rangle \geq 0, \quad u \in N
$$
and
$$
\langle \nabla \cJ(u), V(u) \rangle > 1, \quad u \in \cJ^{c+\eps} \cap \{ u \in X \ : \ \triple{u} \geq \delta \} \cap B_\rho.
$$
As long as $V$ is locally Lipschitz-continuous, the initial value problem
$$
\left\{ \begin{array}{l}
\frac{\partial \eta}{\partial t} (u, t) = - V(\eta(u,t)) \\
\eta(u, 0) = u \in N \supset \cJ^{c+\eps} \cup (X \setminus B_\rho)
\end{array} \right.
$$
has a unique solution $\eta(u, \cdot) : [0, T^+(u)) \rightarrow X$, where $T^+(u) > 0$ is the maximal time of the existence in a positive direction. \\
\textbf{Step 4.} \textit{Properties of the flow $\eta$}. \\
Repeating the proof of \cite[Proposition 2.2]{KS} we see that $\eta$ is $\tau$-continuous. Moreover, for $u \in N$ we get that
$$
\frac{d}{dt} \cJ(\eta(u, t)) = \cJ'(\eta(u, t))(-V(\eta(u,t))) = - \langle \nabla \cJ(\eta(u,t)), V(\eta(u,t)) \rangle \leq 0,
$$
so that $\cJ$ is non-increasing on trajectories $t \mapsto \eta(u, t)$. In particular, if $u \in \cJ^{c+\eps}$, then
$$
\{ \eta(u,t) \ : \ 0 \leq t < T^+(u) \} \subset \cJ^{c+\eps}.
$$
As long as $V$ is sublinear, we see that for $u \in \cJ^{c+\eps}$ we get $T^+(u) = +\infty$. 

Moreover
\begin{align*}
\| \eta(u,t) \| &= \left\| u - \int_0^t V(\eta(u,s)) \, ds \right\| \leq \|u\| + \int_0^t \| V(\eta(u,s)) \| \, ds \\
&\leq \|u\| + \frac{4}{\varepsilon} \int_0^t 1 + \| \eta(u,s)) \| \, ds.
\end{align*}
From the Gronwall's inequality
\begin{align}\label{gronwall}
\| \eta(u,t) \| \leq (1+\|u\|) e^{\frac{4t}{\eps}} - 1.
\end{align}
Set $b := \inf_{\cS_r^+} \cJ$. Then it is clear that 
$$
\sup_{\triple{u} \leq \delta} \cJ < b - \eps.
$$
In particular 
$$
\{ u \in X \ : \ \triple{u} \leq \delta \} \subset \cJ^{b-\eps}.
$$
Hence,
$$
\cJ^{c+\eps}_{b-\eps} \cap B_\rho \subset \cJ^{c+\eps} \cap \{ u \in X \ : \ \triple{u} \geq \delta \} \cap B_\rho.
$$
Thus
\begin{equation}\label{decreasing}
\langle \nabla \cJ(u), V(u) \rangle > 1, \quad u \in \cJ^{c+\eps}_{b-\eps} \cap B_\rho.
\end{equation}
\textbf{Step 5.} \textit{Conclusion.} \\
Fix $u \in \cP$ and $h \in \Gamma(u)$ such that $\sup_{u' \in M(u)} \cJ(h(u',1)) < c + \eps$. 

We claim that $\sup_{u' \in M(u)} \| h(u',1) \| < \infty$. Indeed, $M(u)$ is $\tau$-compact and therefore (h1) implies that $h(M(u),1)$ is $\tau$-compact as well. For each $v \in M(u)$ we can find a $\tau$-open neighbourhood $W_v \subset X$ of $v$ such that $\{ w-h(w,1) \ : \ w \in W_v \}$ is contained in a finite-dimensional subspace of $X$. Since $\{ W_v \}$ is a $\tau$-open covering of $M(u)$, we can choose a finite $\tau$-open subcovering $\{\widetilde{W}_j \}$ of $M(u)$. Then $v-h(v,1) \in \mathbb{V}$ for all $v \in M(u)$ and for some finite-dimensional subspace $\mathbb{V} \subset X$. Then the set $\{ v - h(v,1) \ : \ v \in M(u) \}$ is $\tau$-compact and contained in the finite-dimensional space $\mathbb{V}$, and therefore is bounded. Since $M(u)$ is bounded, so is $h(M(u),1)$.

Set $\rho(u, h) := \left( 1 + \sup_{u' \in M(u)} \| h(u',1) \| \right) e^\frac{4T_0}{\eps} - 1$, where $T_0 := 2 \eps + c - b$. For such $\rho = \rho (u, h)$ we obtain the flow $\eta$ with above conditions. It is clear that for $u' \in M(u)$ we have $h(u',1) \in \cJ^{c+\eps}$. Moreover, from \eqref{gronwall}, we obtain that
$$
\| \eta(h(u',1), t) \| \leq (1+\|h(u',1)\|) e^{\frac{4t}{\eps}} - 1 \leq \rho(u, h)
$$
for $t \in [0, T_0]$. Hence, for $t \in [0, T_0]$ we have that $\eta(h(u',1), t) \in B_\rho$, in particular $h(u',1) \in B_\rho$. Moreover, from \eqref{decreasing}, we see that $\eta(h(u',1), T_0) \in \cJ^{b-\eps}$. Define $g :M(u) \times [0,1] \rightarrow X$ by
$$
g(u', t) := \left\{ \begin{array}{ll}
h(u',2t), & \quad t \in [0,1/2], \\
\eta(h(u',1), T_0 (2t-1)), &\quad t \in [1/2,1].
\end{array} \right.
$$
Then $g \in \Ga(u)$ and $\cJ(g(u',1)) = \cJ( \eta(h(u',1), T_0) ) \leq b - \eps \leq c - \eps$ for any $u' \in M(u)$, which is a contradiction with the definition of $c$. \\
\textbf{Step 6.} \textit{If (A4) holds, then $c \leq \inf_{\cN_\cP} \cJ$.} \\
Observe that if $\cN_\cP = \emptyset$, then $\inf_{\cN_\cP} \cJ = \infty$ and the inequality is trivial. Hence, assume that $\cN_\cP \neq \emptyset$. Take any $u \in \cN_\cP \subset \cP$ and define $h : M(u) \times [0,1] \rightarrow X$ by the formula $h(u', t) = u'$ for $u' \in M(u)$. It is clear that $h$ satisfies conditions (h1)--(h4). Then, (A4) implies that
$$
c \leq \sup_{u' \in M(u)} \cJ(h(u',1)) = \sup_{u' \in M(u)} \cJ(u') = \sup_{tu+v \in M(u)} \cJ(t u + v) \leq \cJ(u),
$$
where $t \geq 0$ and $v \in X^-$. Hence, $c \leq \inf_{\cN_\cP} \cJ$ and the proof is completed.
\end{proof}

\section{Functional setting}\label{sect:functional}

Let
$$
X := \left\{ u \in H^1 (\R^N) \ : \ u \mbox{ is } O(K)\times I \mbox{ invariant and } \int_{\R^N} \frac{u^2}{r^2} dx < +\infty \right\}.
$$
It is classical to check that under (V), the space $X$ has the orthogonal splitting $X = X^+ \oplus X^-$ such that the quadratic form
$$
\int_{\R^N} |\nabla u|^2 + a \frac{u^2}{r^2} + V(x)u^2 \, dx
$$
is positive definite on $X^+$ and negative definite on $X^-$. Hence, we may define norms on $X^+$ and $X^-$ by
$$
\|u^\pm\|^2 := \pm \int_{\R^N} |\nabla u|^2 + a \frac{u^2}{r^2} + V(x)u^2 \, dx, \quad u^\pm \in X^\pm,
$$
We define the product topology on $X$ by
$$
\|u\|^2 := \|u^+\|^2 + \|u^-\|^2,
$$
where $u = u^+ + u^-$, $u^\pm \in X^\pm$. Moreover, projections $X \to X^\pm$ are continuous in $L^q(\R^N)$ (see \cite[Proposition 7]{T}). We will denote by $\kappa \geq 1$ the constant such that
\begin{equation}\label{eq:kappa}
|u^\pm|_q \leq \kappa |u|_q
\end{equation}
for $u \in X$. In view of (V), there is a constant $\mu_0 > 0$ such that
\begin{equation}\label{mu0}
\mu_0 |u|_2 \leq \|u\|, \quad u \in X.
\end{equation}

The energy functional $\cJ : X \rightarrow \R$ associated to \eqref{eq:schroedinger} is of the form
\begin{equation}\label{eq:J}
\cJ(u) := \frac12 \|u^+\|^2 - \frac12 \|u^-\|^2 - \int_{\R^N} F(u) \, dx + \lambda \int_{\R^N} G(u) \, dx
\end{equation}
and it is classical to check that under (F1), (G1) it is of $\cC^1$ class. Note that, for $K=2$, it is not true that $\cC_0^\infty (\R^N) \subset X$. Hence, we say that $u$ is a \textit{weak solution} to \eqref{eq:schroedinger} if $u$ is a critical point of $\cJ$. For $K > 2$ the following inequality
$$
\int_{\R^N} \frac{u^2}{r^2} \, dx \leq \left( \frac{2}{K-2} \right)^2 \int_{\R^N} |\nabla u|^2 \, dx, \quad u \in H^1 (\R^N)
$$
is true (see \cite{BT}). Hence, $\int_{\R^N} \frac{u^2}{r^2} \, dx < +\infty$ for every $u \in H^1 (\R^N)$.

It is also standard to show that the derivative $\cJ'$ is weak-to-weak* continuous, i.e.
$$
\cJ'(u_n)(v) \to \cJ'(u_0)(v)
$$
if $u_n \weakto u_0$ in $X$ and $v \in X$.

We shall also introduce the energy functional and the notion of weak solutions to the curl-curl problem \eqref{eq:maxwell}. We consider the Hilbert space $H^1(\R^3; \R^3)$ and we introduce the energy functional $\cE : H^1(\R^3; \R^3)$ associated with \eqref{eq:maxwell}
\begin{equation}\label{eq:E}
\cE(\mathbf{E}) := \frac12 \int_{\R^3} |\nabla \times \mathbf{E}|^2 + V(x) |\mathbf{E}|^2 \, dx - \int_{\R^3} H(\mathbf{E}) \, dx,
\end{equation}
where we set
$$
H(\mathbf{E}) := \int_0^1 h(t \mathbf{E}) \cdot \mathbf{E} \, dt.
$$
Then $H$ is of $\cC^1$ class and we say that $\mathbf{E}$ is a \textit{weak solution} to \eqref{eq:maxwell} if $\mathbf{E}$ is a critical point of $\cE$.

In what follows we set $\tilde{f}(u) := f(u) - \lambda g(u)$. From \cite[Theorem 1.1]{Bi}, \cite[Theorem 2.1]{GMS} we have the following result.

\begin{Th}\label{th:equiv}
Let $N=3$, $K=2$, $a=1$. Suppose that (V) holds, $\tilde{f}$ is continuous and satisfy
$$
|\tilde{f}(u)| \lesssim |u|+|u|^5, \quad u \in \R.
$$
If $\mathbf{E} \in H^1 (\R^3; \R^3)$ is a weak solution to \eqref{eq:maxwell} of the form \eqref{eq:form}, where $u$ is cylindrically symmetric, then $u \in H^1 (\R^3)$ and $u$ is a weak solution to \eqref{eq:schroedinger}. If $u \in H^1 (\R^3)$ is a cylindrically symmetric, weak solution to \eqref{eq:schroedinger}, then $\mathbf{E}$ given by \eqref{eq:form} lies in $H^1(\R^3; \R^3)$ and is a weak solution to \eqref{eq:maxwell}. Moreover $\divv \mathbf{E}=0$ and $\cE(\mathbf{E})=\cJ(u)$.
\end{Th}

Below we collect some useful properties of nonlinear functions $f$ and $g$. By $(F1)-(F2)$ we deduce that for every $\eps > 0$ there exists $C_{\eps} > 0$ such that
\begin{equation}
\label{growth_f}
	|f(u)| \leq \eps|u| + C_{\eps}|u|^{p-1}.
\end{equation}
Similarly, $(G1)-(G2)$ imply
\begin{equation}
\label{growth_g}
	|g(u)| \leq \eps|u| + C_{G,\eps}|u|^{q-1}.
\end{equation}

By (F3), (F4), (G3), we obtain also the following two conditions
\begin{align}
\label{AR_f}
	0 &\leq qF(u) \leq f(u)u \\
	\label{AR_g}
	 0 &\leq g(u)u \leq qG(u).
\end{align}

\begin{Lem}
For every $\eps > 0$ there exists $C_{F,\eps} > 0$ such that
\begin{equation}
\label{F_below}
	F(u) \geq C_{F,\eps}|u|^q - \eps|u|^2,
\end{equation}
for any $u \in \R^N$, with $ q \in (2,p)$ given in (F3).
\end{Lem}
\begin{proof}
Fix $\varepsilon > 0$ and consider the function $A : \R \setminus \{0\} \rightarrow \R$ given by
$$
A(u) := \frac{F(u) + \varepsilon u^2}{|u|^q}. 
$$
It is clear that $A$ is continuous, $A(u) > 0$ for $u \neq 0$, $\lim_{u\to 0} A(u) = +\infty$, and $\lim_{|u|\to+\infty} A(u) = +\infty$. Hence, $A$ has a positive minimum and it is sufficient to set
$$
C_{F,\eps} := \min_{u \in \R \setminus \{0\}} A(u).
$$
\end{proof}

Observe that, without the loss of generality, we may assume that $C_{G,\eps} \geq C_{F,\eps}$, where $C_{G,\eps}$ is given in \eqref{growth_g} and $C_{F,\eps}$ is given in \eqref{F_below}.

\section{Existence of a Cerami sequence - verification of (A1)--(A3)}\label{sect:3}

We will show that there is a bounded away from zero, Cerami sequence for $\cJ$. Hence, we will check conditions (A1)--(A3) of Theorem \ref{th:abstrExistence} with
$$
\cP := X^+ \setminus \{0\}.
$$
It is clear that (A1), (A2) are satisfied, so we will focus on (A3) and we divide the proof into three steps:
\begin{itemize}[leftmargin=2cm]
\item[\textit{Step 1.}] there is $r > 0$ such that $\inf_{\cS_r^+} \cJ > 0$;
\item[\textit{Step 2.}] for $u \in \cP$, there is radius $R(u) > r$ such that $\sup_{\partial M(u)} \cJ \leq 0$;
\item[\textit{Step 3.}] there is $\delta > 0$ such that $\sup_{\triple{u} \leq \delta} \cJ(u) < \inf_{\cS_r^+} \cJ$.
\end{itemize}

To show \textit{Step 1}, fix $u^+ \in X^+$ and note that, by \eqref{growth_f} and Sobolev embeddings, 
\begin{align*}
\cJ(u^+) \geq \frac12 \|u^+\|^2 - \int_{\R^N} F(u^+) \, dx \geq \frac12 \|u^+\|^2 - \varepsilon C \|u^+\|^2 - \widetilde{C_\varepsilon} \|u^+\|^{p}
\end{align*}
for some $C, \widetilde{C_\eps} > 0$. Choosing sufficiently small $\varepsilon > 0$ and $r > 0$ we easily obtain that
$$
\inf_{\cS_r^+} \cJ \geq  \frac{r^2}{4} > 0.
$$

To check Steps 2 and 3 we need to assume that
\begin{equation}\label{eq:lambda_estimate}
\lambda < \frac{1}{\kappa 2^q} \frac{C_{F,\mu_0 / 8}}{C_{G,\mu_0 / 8}},
\end{equation}
where $C_{F, \mu_0 / 8} \leq C_{G,\mu_0 / 8}$ are given in \eqref{F_below}, \eqref{growth_g}.

Fix $u \in \cP$. We will show that under \eqref{eq:lambda_estimate}, $\sup_{\partial M(u)} \cJ \leq 0$ for sufficiently large $R(u)$. Note that \eqref{eq:lambda_estimate} implies that $\lambda \leq 1$. Take $u_n \in \R^+ u \oplus X^-$. Then $u_n = t_n u + u_n^-$ for some $t_n \geq 0$, $u_n^- \in X^-$. Without the loss of generality we may assume that $\|u\|=1$. Using \eqref{growth_g}, \eqref{F_below} and \eqref{mu0} we get
\begin{align*}
\cJ(u_n) &= \cJ(t_n u + u_n^-) = \frac12 t_n^2 - \frac12 \|u_n^-\|^2 - \int_{\R^N} F(u_n) \, dx + \lambda \int_{\R^N} G(u_n) \, dx \\
&\leq \frac12 t_n^2 - \frac12 \|u_n^-\|^2 - C_{F,\eps} |u_n|^q_q + \eps |u_n|^2_2 + \lambda \eps |u_n|_2^2  + \lambda C_{G,\eps} |u_n|_q^q \\
&\leq \left( -\frac12 + \frac{\eps+\lambda \eps}{\mu_0} \right) \|u_n^-\|^2 + \left( \frac12 + \frac{\eps+\lambda \eps}{\mu_0} \right) t_n^2 - C_{F,\eps} |t_n u+u_n^-|_q^q + \lambda C_{G,\eps} |t_n u+u_n^-|_q^q \\
&\leq \left( -\frac12 + \frac{2\eps}{\mu_0} \right) \|u_n^-\|^2 + \left( \frac12 + \frac{2\eps}{\mu_0} \right) t_n^2 - C_{F,\eps} |t_n u+u_n^-|_q^q + \lambda C_{G,\eps} |t_n u+u_n^-|_q^q.
\end{align*}
Choosing $\eps := \frac{\mu_0}{8}$ and using \eqref{eq:kappa}, we get
\begin{align*}
&\quad \left( -\frac12 + \frac{2\eps}{\mu_0} \right) \|u_n^-\|^2 + \left( \frac12 + \frac{2\eps}{\mu_0} \right) t_n^2 - C_{F,\eps} |t_n u+u_n^-|_q^q + \lambda C_{G,\eps} |t_n u+u_n^-|_q^q \\
&\leq -\frac14 \|u_n^-\|^2 + \frac34 t_n^2 - C_{F,\mu_0 / 8} |t_n u + u_n^-|_q^q + \lambda C_{G, \mu_0 / 8} |t_n u + u_n^-|_q^q \\
&\leq -\frac14 \|u_n^-\|^2 + \frac34 t_n^2 - \frac{C_{F,\mu_0 / 8}}{2 \kappa}  |u|_q^q t_n^q  - \frac{C_{F,\mu_0 / 8}}{2\kappa} |u_n^-|_q^q \\
&\quad + 2^{q-1} \lambda C_{G,\mu_0 / 8} |u|_q^q t_n^q + 2^{q-1}\lambda C_{G,\mu_0 / 8} |u_n^-|_q^q \\
&= -\frac14 \|u_n^-\|^2 + \frac34 t_n^2 + \left( 2^{q-1} \lambda C_{G,\mu_0 / 8} - \frac{C_{F,\mu_0 / 8}}{2 \kappa} \right) |u|_q^q t_n^q \\
&\quad  + \left( 2^{q-1}\lambda C_{G,\mu_0 / 8} - \frac{C_{F,\mu_0 / 8}}{2\kappa} \right) |u_n^-|_q^q.
\end{align*}
Using \eqref{eq:lambda_estimate} we finally arrive at
$$
\cJ(u_n) \leq  -\frac14 \|u_n^-\|^2 + \frac34 t_n^2 + \left( 2^{q-1} \lambda C_{G,\mu_0 / 8} - \frac{C_{F,\mu_0 / 8}}{2 \kappa} \right) |u|_q^q t_n^q.
$$
Hence, $\cJ(u_n) \to -\infty$ as $\|t_n u + u_n^-\| \to +\infty$. In particular, for $t_n = 0$, we get $\cJ(u_n^-) \leq 0$. Thus $\sup_{X^-} \cJ \leq 0$ and $\sup_{\partial M(u)} \cJ \leq 0$ for sufficiently large $R(u)$, and the proof of \textit{Step 2} is completed. To get \textit{Step 3} we use \eqref{growth_g}, \eqref{F_below} and \eqref{mu0}, and compute
\begin{align*}
\cJ(u) &\leq \frac12 \|u^+\|^2 - \frac12 \|u^-\|^2 - C_{F,\eps} |u|_q^q + \eps |u|_2^2 + \lambda \eps |u|_2^2 + \lambda C_{G,\eps} |u|_q^q \\
&\leq \left( \frac12 + \frac{\eps + \lambda \eps}{\mu_0} \right) \|u^+\|^2 - \left(\frac12 - \frac{\eps+\lambda \eps}{\mu_0} \right) \|u^-\|^2 - C_{F,\eps} |u|_q^q + \lambda C_{G,\eps} |u|_q^q.
\end{align*}
Note that \eqref{eq:lambda_estimate} implies that $\lambda < \frac{C_{F,\mu_0 / 8}}{C_{G,\mu_0 / 8}} \leq 1$ and therefore, for $\eps := \frac{\mu_0}{8}$,
\begin{align*}
\cJ(u) &\leq \left( \frac12 + \frac{2\eps}{\mu_0} \right) \|u^+\|^2 - \left(\frac12 - \frac{2\eps}{\mu_0} \right) \|u^-\|^2 \leq \frac34 \|u^+\|^2 \leq \frac34 \triple{u}^2 \to 0
\end{align*}
as $\triple{u} \to 0$, and the proof of \textit{Step 3} is completed.

\begin{Rem}
In the case $\lambda = 0$ we are able to show also the condition (A4), which follows from the inequality (see e.g. \cite[Lemma 3.2]{M2})
\begin{equation}\label{eq:J_ineq}
\cJ(u) \geq \cJ(tu+v) - \cJ'(u) \left( \frac{t^2-1}{2}u + tv \right), \quad u \in X, \ v \in X^-, \ t \geq 0,
\end{equation}
and obtain an additional estimate of $c$ in terms of the Nehari-Pankov manifold. Indeed, if $\lambda = 0$ the nonlinear part of the functional is nonnegative, we take $\cP := X \setminus X^-$ and then $\cN_\cP = \cN$. Hence, for any $t \geq 0$, $u \in \cN$ and $v \in X^-$ we get 
$$
\cJ'(u) \left( \frac{t^2-1}{2}u + tv \right) = \frac{t^2-1}{2} \cJ'(u)(u) + t \cJ'(u)(v) = 0
$$
and \eqref{eq:J_ineq} leads to (A4). Then, from Theorem \ref{th:abstrExistence}, we obtain also that $c \leq \inf_\cN \cJ$.
\end{Rem}

\section{Boundedness of Cerami-type sequences}\label{sect:5}

Now, we are going to discuss the boundedness of a Cerami sequence for $\cJ$.

\begin{Lem}\label{lem:cerami_bdd}
Suppose that $\lambda > 0$ and $\rho > 0$ in (F5) are sufficiently small. Let $(u_n) \subset X$ satisfy
$$
\cJ(u_n) \leq \beta, \quad (1+\|u_n\|) \cJ'(u_n) \to 0
$$
for some $\beta \in \R$. Then $(u_n)$ is bounded in $X$. In particular, any Cerami sequence for $\cJ$ is bounded.
\end{Lem}

\begin{proof}
Suppose by contradiction that $\|u_n\|\to\infty$ and note that
$$
\|u_n\|^2 = \|u_n^+\|^2 + \|u_n^-\|^2 = \int_{\R^N} (f(u_n) - \lambda g(u_n)) (u_n^+ - u_n^-) \, dx + o(1),
$$
and
\begin{align*}
&\quad \int_{\R^N} (f(u_n) - \lambda g(u_n)) (u_n^+ - u_n^-) \, dx\\
&= \underbrace{\int_{|u_n| < \rho} (f(u_n) - \lambda g(u_n)) (u_n^+ - u_n^-) \, dx}_{=:I_1} + \underbrace{\int_{|u_n| \geq \rho} (f(u_n) - \lambda g(u_n)) (u_n^+ - u_n^-) \, dx}_{=:I_2}.
\end{align*}
To estimate $I_1$ we fix $\eps > 0$, and for some constant $C_\eps > 0$ we obtain
\begin{align*}
I_1 &\leq \int_{|u_n| < \rho} |f(u_n)-\lambda g(u_n)| \left| u_n^+ - u_n^- \right| \, dx  \\
&\leq \eps (1+\lambda) \int_{|u_n| < \rho} |u_n| \left| u_n^+ - u_n^- \right| \, dx + C_\eps \int_{|u_n|<\rho} |u_n|^{p-1} \left| u_n^+ - u_n^- \right| \, dx \\
&\quad  + C_\eps \lambda \int_{|u_n|<\rho} |u_n|^{q-1} \left| u_n^+ - u_n^- \right| \, dx \\
&\leq \left( \eps(1+\lambda) + C_\eps \rho^{p-2} + \lambda C_\eps \rho^{q-2} \right) \int_{|u_n| < \rho} |u_n| \left| u_n^+ - u_n^- \right| \, dx \\
&= \left( \eps(1+\lambda) + C_\eps \rho^{p-2} + \lambda C_\eps \rho^{q-2} \right) \int_{|u_n| < \rho} |u_n^+|^2 - |u_n^-|^2 \, dx \\
&\leq \left( \eps(1+\lambda) + C_\eps \rho^{p-2} + \lambda C_\eps \rho^{q-2} \right) \int_{|u_n| < \rho} |u_n^+|^2 \, dx \\
&\leq \frac{1}{\mu_0} \left( \eps(1+\lambda) + C_\eps \rho^{p-2} + \lambda C_\eps \rho^{q-2} \right) \|u_n^+\|^2 \leq \frac{1}{\mu_0} \left( \eps(1+\lambda) + C_\eps \rho^{p-2} + \lambda C_\eps \rho^{q-2} \right) \|u_n\|^2.
\end{align*}

To estimate $I_2$, we observe that (F1), (F4), (F5), (G1), (G3) imply that $$\{ |u| \geq \rho \} \ni u \mapsto \frac{g(u)}{f(u)} \in \R$$ is well-defined, nonincreasing, nonnegative and even. Hence,
$$
\frac{g(\rho)}{f(\rho)} \geq \left| \frac{g(u)}{f(u)} \right|, \quad |u| \geq \rho.
$$
Hence,
\begin{align*}
I_2 &\leq \int_{|u_n| \geq \rho} |f(u_n) - \lambda g(u_n)| \left( |u_n^+|+|u_n^-| \right) \, dx \\
&= \int_{|u_n| \geq \rho} |f(u_n)| \left| 1 - \lambda \frac{g(u_n)}{f(u_n)} \right| \left( |u_n^+| + |u_n^-| \right) \, dx \\
&\leq \left(1+\lambda \frac{g(\rho)}{f(\rho)} \right) \int_{|u_n| \geq \rho} |f(u_n)| \left( |u_n^+| + |u_n^-| \right) \, dx \\
&\leq C  \left(1+\lambda \frac{g(\rho)}{f(\rho)} \right) \int_{|u_n| \geq \rho}  |u_n|^{p-1}  \left( |u_n^+| + |u_n^-| \right) \, dx \\
&\leq C  \left(1+\lambda \frac{g(\rho)}{f(\rho)} \right)  2\kappa \int_{\R^N} |u_n|^{p}  \, dx.
\end{align*}
To estimate the $L^p$-norm of $u_n$ we observe the following
\begin{align*}
\beta + o(1) &\geq \cJ(u_n) - \frac12 \cJ'(u_n)(u_n) = \int_{\R^N} \Phi(u_n) \, dx,
\end{align*}
where we set
$$
\Phi(u) := \frac12 f(u)u - F(u) + \lambda G(u) - \frac{\lambda}{2} g(u)u
$$
for a simplicity of notation. Using (F5), \eqref{AR_f}, \eqref{AR_g} and choosing $\lambda$ so small that $1-\lambda \frac{g(\rho)}{f(\rho)} > 0$, we get
\begin{align*}
&\quad \beta + o(1) + \int_{|u_n| < \rho} |\Phi(u)| \, dx \geq \beta + o(1) - \int_{|u_n| < \rho} \Phi(u) \, dx \\
&= \beta + o(1) - \int_{\R^N} \Phi(u_n) \, dx + \int_{|u_n| \geq \rho} \Phi(u_n) \, dx \geq \int_{|u_n| \geq \rho} \Phi(u_n) \, dx \\
&= \int_{|u_n| \geq \rho} \left[ \frac12 f(u_n)u_n - F(u_n) + \lambda G(u_n) - \frac12 g(u_n)u_n \right] \, dx \\
&\geq \left( \frac12 - \frac1q \right) \int_{|u_n| \geq \rho}  f(u_n)u_n - \lambda g(u_n)u_n \, dx \\
&= \left( \frac12 - \frac1q \right) \int_{|u_n| \geq \rho} \left(1 - \lambda \frac{g(u_n)}{f(u_n)} \right) f(u_n) u_n \, dx \\
&\geq \left( \frac12 - \frac1q \right) \left(1 - \lambda \frac{g(\rho)}{f(\rho)} \right) \int_{|u_n| \geq \rho}  f(u_n) u_n \, dx \\
&\gtrsim \left( \frac12 - \frac1q \right) \left(1 - \lambda \frac{g(\rho)}{f(\rho)} \right) \int_{|u_n| \geq \rho} |u_n|^p \, dx.
\end{align*}
Thus
\begin{equation}\label{1}
\int_{|u_n| \geq \rho} |u_n|^p \, dx \leq C \left(1 - \lambda \frac{g(\rho)}{f(\rho)} \right)^{-1} \left( \beta + \int_{|u_n| < \rho} |\Phi(u_n)| \, dx \right) + o(1)
\end{equation}
for some constant $C > 0$ independent of $n$, $\lambda$ and $\rho$. Therefore, using \eqref{1} we get
\begin{align*}
I_2 &\leq \underbrace{ C \left(1+\lambda \frac{g(\rho)}{f(\rho)} \right)   2\kappa}_{=:D(\lambda, \rho, \eps)} \left( \int_{|u_n| < \rho} |u_n|^{p}  \, dx + \int_{|u_n| \geq \rho} |u_n|^{p}  \, dx\right) \\
&\leq D(\lambda, \rho, \eps) \left( \int_{|u_n| < \rho} |u_n|^{p}  \, dx + C \left(1 - \lambda \frac{g(\rho)}{f(\rho)} \right)^{-1} \left( \beta + \int_{|u_n| < \rho} |\Phi(u_n)| \, dx \right)  \right) + o(1) \\
&\leq D(\lambda,\rho,\eps) \left( \rho^{p-2} |u_n|_2^2 + C \left(1 - \lambda \frac{g(\rho)}{f(\rho)} \right)^{-1} \left( \beta + \sup_{|t| \leq \rho} \frac{|\Phi(t)|}{t^2} |u_n|_2^2 \right) \right)+ o(1) \\
&\leq D(\lambda,\rho,\eps) \left( \frac{\rho^{p-2}}{\mu_0} \|u_n\|^2 + \frac{C \beta}{\left(1 - \lambda \frac{g(\rho)}{f(\rho)} \right)}  + \frac{C}{\left(1 - \lambda \frac{g(\rho)}{f(\rho)} \right) \mu_0} \sup_{|t| \leq \rho} \frac{|\Phi(t)|}{t^2} \|u_n\|^2 \right)+ o(1) \\
&\leq D(\lambda,\rho,\eps) \left( \frac{\rho^{p-2}}{\mu_0} + \frac{C}{\left(1 - \lambda \frac{g(\rho)}{f(\rho)} \right) \mu_0} \sup_{|t| \leq \rho} \frac{|\Phi(t)|}{t^2} \right) \|u_n\|^2 + \widetilde{C}
\end{align*}
for some $\widetilde{C} = \widetilde{C}(\lambda, \rho, \eps) > 0$. Finally
\begin{align*}
&\quad \|u_n\|^2 = I_1 + I_2 + o(1)\leq \frac{K}{\mu_0} \|u_n\|^2 + \widetilde{C},
\end{align*}
where 
\begin{align*}
&\quad K := \eps (1+\lambda) + C_\eps \rho^{p-2} + \lambda C_\eps \rho^{q-2}   + D(\lambda,\rho,\eps) \left( \rho^{p-2} + \frac{C}{1 - \lambda \frac{g(\rho)}{f(\rho)} } \sup_{|t| \leq \rho} \frac{|\Phi(t)|}{t^2} \right) \\
&= \eps (1+\lambda) + C_\eps \rho^{p-2} + \lambda C_\eps \rho^{q-2}   + C 2\kappa  \rho^{p-2} + C 2\kappa \lambda \frac{g(\rho)}{f(\rho)}  \rho^{p-2} + C  \frac{1+\lambda \frac{g(\rho)}{f(\rho)}}{1 - \lambda \frac{g(\rho)}{f(\rho)}} \sup_{|t| \leq \rho} \frac{|\Phi(t)|}{t^2} 
\end{align*}
Hence, the proof is completed if $K < \mu_0$. It is clear that $\lim_{t \to 0} \frac{|\Phi(t)|}{t^2} = 0$, hence $\sup_{|t| \leq \rho} \frac{|\Phi(t)|}{t^2}$ can be arbitrarily small for small $\rho$. We recall that we already need \eqref{eq:lambda_estimate}, and in particular $\lambda \leq 1$. Hence,
$$
K \leq 2\eps  + C_\eps \rho^{p-2} +  C_\eps \rho^{q-2}   + C 2\kappa  \rho^{p-2} + C 2\kappa \lambda \frac{g(\rho)}{f(\rho)}  \rho^{p-2} + C  \frac{1+\lambda \frac{g(\rho)}{f(\rho)}}{1 - \lambda \frac{g(\rho)}{f(\rho)}} \sup_{|t| \leq \rho} \frac{|\Phi(t)|}{t^2} .
$$
Fix $\varepsilon < \frac{\mu_0}{12}$. Now we choose $\rho > 0$ so small that
$$
C_\eps \rho^{p-2} +  C_\eps \rho^{q-2}   + C 2\kappa  \rho^{p-2} + 2C \sup_{|t| \leq \rho} \frac{|\Phi(t)|}{t^2} < \frac{2 \mu_0}{3}.
$$
Choosing $\lambda$ so small that
$$
C 2\kappa \lambda \frac{g(\rho)}{f(\rho)}  \rho^{p-2} < \frac{\mu_0}{6}, \quad 0 \leq \frac{1+\lambda \frac{g(\rho)}{f(\rho)}}{1 - \lambda \frac{g(\rho)}{f(\rho)}} \leq 2.
$$
we obtain $K < \mu_0$.
\end{proof}

\begin{Prop}\label{prop:bdd}
Let $\beta \in \R$. There is a constant $M_\beta > 0$ such that for every sequence $(u_n) \subset X$ satisfying
$$
\cJ(u_n) \leq \beta, \quad (1+\|u_n\|) \cJ'(u_n) \to 0
$$
there holds
$$
\limsup_{n\to\infty} \| u_n\| \leq M_\beta.
$$
\end{Prop}

\begin{proof}
Suppose by contradiction that there is $\beta \in \R$ and for any $k \geq 1$ we find a sequence $(u_n^k) \subset X$ such that
$$
\cJ(u_n^k) \leq \beta, \quad (1+\|u_n^k\|) \cJ'(u_n^k) \to 0
$$
and 
$$
\limsup_{n\to\infty} \| u_n^k\| \geq k.
$$
Let $n(k)$ be a number such that $\|u_{n(k)}^k\| \geq k-1$. Without loss of generality we may assume that $n(k)$ is increasing with respect to $k$. Then the sequence $\left( u_{n(k)}^k \right)_k$ satisfies assumptions of Lemma \ref{lem:cerami_bdd}, but is unbounded - a contradiction.
\end{proof}

\section{Existence of a nontrivial solution}

In order to prove Theorem \ref{th:main1}, we need the following concentration-compactness principle in the spirit of Lions, see \cite[Corollary 3.2, Remark 3.3]{M3}.

\begin{Lem}\label{lem:lions}
Suppose that $(u_n) \subset X$ is bounded and for all $R > 0$ the following vanishing condition 
\begin{equation}\label{vanishing}
\lim_{n\to+\infty} \sup_{z \in \R^{N-K}} \int_{B( (0,z), R)} |u_n|^2 \, dx =0
\end{equation}
holds. Then,
$$
\int_{\R^N} |\Psi(u_n)| \, dx \to 0 \ as \ n \to +\infty
$$
for any continuous function $\Psi : \R \rightarrow \R$ satisfying
$$
\lim_{s \to 0} \frac{\Psi(s)}{s^2} = \lim_{|s|\to+\infty} \frac{\Psi(s)}{s^{2^*}} = 0.
$$
\end{Lem}

As an easy consequence of Lemma \ref{lem:lions} we obtain the following.

\begin{Prop}\label{prop:lions}
Suppose that a bounded sequence $(u_n) \subset X$ satisfies \eqref{vanishing} for every $R > 0$. Then
$$
\int_{\R^N} \tilde{f}(u_n)u_n^{\pm} \, dx \to 0.
$$
\end{Prop}

\begin{proof}
Fix $\eps > 0$. Then, for some $C_\eps > 0$
\begin{align*}
\left| \int_{\R^N} \tilde{f}(u_n)u_n^{\pm} \, dx \right| &\leq \eps \int_{\R^N} |u_n u_n^\pm| \, dx + C_\eps \left( \int_{\R^N} |u_n|^{p-1} |u_n^\pm| \, dx + \int_{\R^N} |u_n|^{q-1} |u_n^\pm| \, dx \right) \\
&\leq \eps |u_n|_2 |u_n^\pm |_2 + C_\eps \left( |u_n^\pm|_p \left( \int_{\R^N} |u_n|^p \,dx \right)^{\frac{p-1}{p}} + |u_n^\pm|_q \left( \int_{\R^N} |u_n|^q \,dx \right)^{\frac{q-1}{q}} \right) \\
&\lesssim \eps + C_\eps \left( \left( \int_{\R^N} |u_n|^p \,dx \right)^{\frac{p-1}{p}} + \left( \int_{\R^N} |u_n|^q \,dx \right)^{\frac{q-1}{q}} \right).
\end{align*}
From Lemma \ref{lem:lions} we get $\int_{\R^N} |u_n|^p \,dx \to 0$ and $\int_{\R^N} |u_n|^q \,dx \to 0$, and therefore
$$
\int_{\R^N} \tilde{f}(u_n)u_n^{\pm} \, dx \to 0.
$$
\end{proof}

\begin{proof}[Proof of Theorem \ref{th:main1}]
By Section \ref{sect:3}, the functional $\cJ$ satifies assumtpions (A1)-(A3), hence we find a sequence $(u_n)$ satisfying \eqref{Cer}. Moreover, by Proposition \ref{prop:bdd} the sequence $(u_n)$ is bounded. Hence, up to a subsequence, there exists $u_0 \in X$ such that $u_n \weakto u_0$. Suppose that \eqref{vanishing} holds for every radius $R > 0$. Then, by Lemma \ref{lem:lions} we have that $u_n \to 0$ in $L^t(\R^N)$ for any $2 < t < 2^*$. From \eqref{Cer} and Proposition \ref{prop:lions} we have that
\begin{align*}
	o(1)&=\cJ'(u_n)u_n^+ = \|u_n^+\|^2 - \int_{\R^N}\tilde{f}(u_n)u_n^+ \, dx = \|u_n^+\|^2 + o(1)
\end{align*}
and therefore $\|u_n^+\| \to 0$. Similarly we obtain that $\|u_n^-\| \to 0$. Thus $\triple{u_n} \leq \|u_n\| \to 0$ and we reach a contradiction, since $\triple{u_n} \geq \frac{\delta}{2}$. Hence, there is $R > 0$ and a sequence $(z_n)_n \subset \Z^{N-K}$ such that
$$
\liminf_{n \to +\infty}\int_{B(0, R)} |v_n|^2 \, dx > 0,
$$
where $v_n:=u_n(\cdot,\cdot-z_n)$. Moreover $\|v_n\| = \|u_n\|$ so that $(v_n)$ is bounded and $v_n \weakto v_0 \neq 0$. $\{ I_K \} \times \Z^{N-K}$-invariance of $\cJ$ implies that $(v_n)$ also satisfies $(1+\|v_n\|)\cJ'(v_n) \to 0$ and by the weak-to-weak* continuity of $\cJ'$ we obtain that $\cJ'(v_0) = 0$, and the proof is completed.
\end{proof}

\begin{proof}[Proof of Theorem \ref{th:main2}]
The statement follows directly from Theorem \ref{th:main1} and Theorem \ref{th:equiv}.
\end{proof}

\begin{Prop}\label{prop:energy}
The total electromagnetic energy $\cL(t)$ given by \eqref{eq:electroen} of the solution $\mathbf{E}$ found in Theorem \ref{th:main2}, is finite and does not depend on $t$.
\end{Prop}
\begin{proof}
Using the equivalence (Theorem \ref{th:equiv}) and constitutive relations, we get 
\begin{align*}
\cL(t)&=\frac12 \int_{\R^3} \cE\cD + \cB\cH \, dx = \frac12 \int_{\R^3} \cE\cD + \cB \cB \, dx \\
&= \dfrac{1}{2\omega^2} \int_{\R^3} \left(-V(x)|\mathbf{E}|^2+h(\mathbf{E})\mathbf{E}\right)\cos^2(\omega t) + |\nabla \times \mathbf{E}|^2 \sin^2(\omega t)  \, dx \\
&= \dfrac{1}{2\omega^2} \int_{\R^3} \left(-V(x)|u|^2+\tilde{f}(u)u\right)\cos^2(\omega t) + \left(|\nabla u|^2 + \frac{u^2}{r^2} \right) \sin^2(\omega t)  \, dx.
\end{align*}
Since $u \in X$, $|\cL(t)| < +\infty$. To show that $\cL(t)$ does not depend on $t$ we compute that
\begin{align*}
\frac{d}{dt} \cL(t) &= \frac{d}{dt} \dfrac{1}{2\omega^2} \int_{\R^3} \left(-V(x)|u|^2+\tilde{f}(u)u\right)\cos^2(\omega t) + \left(|\nabla u|^2 + \frac{u^2}{r^2} \right) \sin^2(\omega t)  \, dx \\
&=\dfrac{\sin(\omega t) \cos(\omega t)}{\omega} \int_{\R^3} |\nabla u|^2 + \frac{u^2}{r^2}  + V(x)|u|^2-\tilde{f}(u)u   \, dx = 0.
\end{align*}
\end{proof}

\section*{Acknowledgements}

The authors would like to thank anonymous referees for valuable comments and especially for indicating a gap in the assumption (A3) and providing simplifications of the proof of Theorem \ref{th:abstrExistence} in the original version of the manuscript.

Bartosz Bieganowski was partially supported by the National Science Centre, Poland (Grant No. 2017/25/N/ST1/00531). Federico Bernini started working on this project during his scientific internship in Nicolaus Copernicus University in Toru\'{n}, supported by the Project PROM. "Project PROM - International scolarship exchange of PhD candidates and academic staff" is co-funded from the European Social Fund as part of the Program Knowledge Education Development, a non-contest project entitled "International scolarship exchange of PhD candidates and academic staff", agreement number POWR.03.03.00-00-PN/13/18.

\end{document}